\documentclass{amsart}
\usepackage{amssymb,amsmath,amsthm}
\usepackage[shortalphabetic]{amsrefs}
\usepackage{mathrsfs}
\usepackage{array}
\usepackage{caption}
\usepackage[dvipsnames]{xcolor}
\usepackage{tikz-cd}
\tikzcdset{
    cong/.style={"\cong" {sloped, description, yshift=0pt,#1}, phantom},
    simeq/.style={"\simeq" {sloped, description, yshift=0pt,#1}, phantom}
}
\usepackage[shortlabels]{enumitem}
\usepackage{upgreek}
\usepackage{subcaption}
\usepackage{wrapfig}
\usepackage{todonotes}

\usepackage{longtable,booktabs}

\definecolor{mgreen}{RGB}{0,150,0}
\definecolor{mblue}{RGB}{65,105,225}
\definecolor{bgreen}{rgb}{.1,.7,.1}
\definecolor{dgreen}{rgb}{.05,.3,.05}

\include{scrload}

\usepackage[colorlinks=true,linkcolor=blue,citecolor=mblue]{hyperref}

\usepackage[nameinlink,capitalise,noabbrev]{cleveref}

\crefname{equation}{}{}

\newtheorem{thm}{Theorem}[section]
\newtheorem{cor}{Corollary}[section]
\newtheorem{prop}{Proposition}[section]

\theoremstyle{definition}

\newtheorem{eg}{Example}[section]
\newtheorem{notn}{Notation}[section]
\newtheorem{rmk}[thm]{Remark}
\newtheorem{assumpt}[thm]{Assumption}

\newenvironment{pf}{\begin{proof}}{\end{proof}}

\newcommand{\cA}{\ensuremath{\mathcal{A}}}
\newcommand{\C}{\ensuremath{\mathbb{C}}}

\newcommand{\F}{\ensuremath{\mathbb{F}}}

\newcommand{\rH}{\ensuremath{\mathrm{H}}}

\newcommand{\bM}{\ensuremath{\mathbb{M}_2}}

\newcommand{\R}{\ensuremath{\mathbb{R}}}
\newcommand{\cY}{\ensuremath{\mathcal{Y}}}

\newcommand{\bN}{{\bf N}}
\newcommand{\id}{\mathrm{id}}

\newcommand{\smsh}{\wedge}

\newcommand{\rtarr}{\longrightarrow}

\newcommand{\iso}{\cong}

\newcommand{\hsf}{\mathsf{h}}
\newcommand{\mr}[1]{\mathrm{#1}}

\newcommand{\Sp}{\mathbf{Sp}} 
\newcommand{\Betti}{\mr{\upbeta}}
\newcommand{\Ctwo}{{\mr{C}_2}}
\newcommand{\cYR}{\cY^\R}
\newcommand{\mtau}{{\tau}}
\newcommand{\mrho}{{\rho}}
\newcommand{\btau}{{\uptau}}
\newcommand{\bxi}{{\upxi}}

\newcommand{\bMR}{\bM^\R}
\newcommand{\dual}{\vee}
\newcommand{\conj}{{\sf c}}
\newcommand{\GFixFunctor}{\Phi}
\newcommand{\GFix}[1]{\GFixFunctor(#1)}

\usepackage[scr=boondoxo]{mathalpha}
\newcommand{\jay}{\mathscr{j}}

\DeclareMathOperator{\Sq}{Sq}

\newcommand{\tens}{\mathbin{\otimes}}

\newcommand{\cAR}{\cA^\R}
\newcommand{\cAC}{\cA^\C}

\newcommand{\HF}{\mathbf{H}\mathbb{F}_2}
\newcommand{\HM}{\mathbf{H}_\R\mathbb{F}_2}

\newcommand{\ourcolor}{blue}
\newcommand{\Ddoubleprimecolor}{orange}
\newcommand{\mandalaradius}{1.75pt}

\newcolumntype{C}{>{$}c<{$}}
\newcolumntype{L}{>{$}l<{$}}
\newcolumntype{R}{>{$}r<{$}}
\newcolumntype{H}{>{\setbox0=\hbox\bgroup$}c<{$\egroup}@{}}

\numberwithin{equation}{section}
\numberwithin{figure}{section}
\numberwithin{table}{section}

\begin{document}
\title[The $\cAR$-module structure on $\mathbb{R}$-motivic Spanier-Whitehead duals]{On the  Steenrod module structure of $\mathbb{R}$-motivic Spanier-Whitehead duals \ }
\author{Prasit Bhattacharya}
\address{Department of Mathematics,
New Mexico State University,
Las Cruces, NM 88003, USA} 
\email{prasit@nmsu.edu}
\author{Bertrand {J}. Guillou}
\address{Department of Mathematics\\ University of Kentucky\\
Lexington, KY 40506, USA}
\email{bertguillou@uky.edu}
\author{Ang Li}
\address{Department of Mathematics\\ University of California\\ Santa Cruz, CA 95064, USA}
\email{ali169@ucsc.edu}
\thanks{
Guillou was supported by NSF grant DMS-2003204.
}
\thanks{
 \ Bhattacharya is supported by NSF grant DMS-2305016.
}



\begin{abstract}
The $\mathbb{R}$-motivic cohomology of an $\R$-motivic spectrum  is a module over the $\mathbb{R}$-motivic Steenrod algebra $\mathcal{A}^{\mathbb{R}}$. 
In this paper, we describe how to recover the $\R$-motivic cohomology of the Spanier-Whitehead dual $\mathrm{DX}$ of an $\R$-motivic finite complex $\mr{X}$, as an $\mathcal{A}^{\mathbb{R}}$-module, given the $\mathcal{A}^{\mathbb{R}}$-module structure on the cohomology of $\mr{X}$. As an application, we show that 16 out of 128 different  $\mathcal{A}^{\mathbb{R}}$-module structures on $\mathcal{A}^{\mathbb{R}}(1):= \langle \mr{Sq}^1, \mr{Sq}^2 \rangle$ are self-dual. 
\end{abstract}


\maketitle

\section{Introduction}  
\noindent
Given a finite cell complex $\mr{X}$, it is useful to determine its Spanier-Whitehead dual $\mr{DX}$, which is the dual of the suspension spectrum $\Sigma^\infty \mr{X}$ in the stable homotopy category of spectra. For instance, the mod 2 cohomology of $\mr{DX}$,
as  a (left) module over the mod 2 Steenrod algebra $\cA$, is an input of Adams spectral sequences computing homotopy class of maps out of $\mr{X}$. An interesting case is when $\mr{X}$ is self-dual, as it leads to additional symmetries \cite{MR} often useful for computational purposes.

\medskip
\noindent
In the classical case, the $\cA$-module structure on $\rH^*(\mr{DX}; \mathbb{F}_2) \cong \mr{H}_*(\mr{X}, \mathbb{F}_2)$  is determined easily by the standard formula \Cref{eqn:D''}  which involves  the Kronecker pairing and  the antiautomorphism $\chi\colon\cA \longrightarrow \cA$ of the Steenrod algebra.  However, one should not expect an $\R$-motivic generalization of the standard formula  because,  
first,
the cohomology of the Spanier Whitehead dual of an $\R$-motivic finite complex $\mr{X}$ is not always the linear dual of the cohomology of $\mr{X}$,  as the coefficient ring $\bMR$  is not a field (see \eqref{M2}). 
Second,  the  $\R$-motivic Steenrod algebra is not known to support an antiautomorphism (but see \cref{sec:Lift}). 
 
 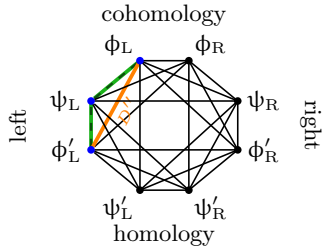
\begin{wrapfigure}{l}{.48\textwidth}
\vspace{-2ex}
\begin{center}
\begin{tikzpicture}
\begin{scope}[semithick, every node/.style={sloped,allow upside down}, scale=0.65,font=\small]
\draw[line width = .5mm,bgreen] (0,0)  node[inner sep=0] (v00) {} -- (0,1) node[inner sep=0] (v01) {};
\draw (3,0)  node[inner sep=0] (v30) {} -- (3,1) node[inner sep=0] (v31) {};
\draw (1,1.825)  node[inner sep=0] (v12) {} -- (2,1.825) node[inner sep=0] (v22) {};
\draw (1,-0.825)  node[inner sep=0] (v1n1) {} -- (2,-0.825) node[inner sep=0] (v2n1) {};
\draw[line width = .5mm,bgreen] (v01) -- (v12);
\draw[line width = .3mm,dashed, dgreen] (v01) -- (v12);
\draw[line width = .3mm,dashed, dgreen] (v00) -- (v01);
\draw (v22) -- (v31);
\draw (v00) -- (v1n1);
\draw (v30) -- (v2n1);
\draw (v00) -- (v2n1);
\draw (v00) -- (v30);
\draw (v00) -- (v22);
\draw[\Ddoubleprimecolor, line width = .5mm] (v00) to node[pos=0.5,below={-0.5ex}] {\tiny$D''$} (v12);
\draw (v01) -- (v22);
\draw (v01) -- (v31);
\draw (v01) -- (v1n1);
\draw (v01) -- (v2n1);
\draw (v31) -- (v12);
\draw (v31) -- (v1n1);
\draw (v31) -- (v2n1);
\draw (v30) -- (v1n1);
\draw (v30) -- (v12);
\draw (v30) -- (v22);
\draw (v12) -- (v1n1);
\draw (v22) -- (v2n1);
\filldraw[\ourcolor] (v00) circle (\mandalaradius);
\filldraw[\ourcolor] (v01) circle (\mandalaradius);
\filldraw (v30) circle (\mandalaradius);
\filldraw (v31) circle (\mandalaradius);
\filldraw[\ourcolor] (v12) circle (\mandalaradius);
\filldraw (v22) circle (\mandalaradius);
\filldraw (v1n1) circle (\mandalaradius);
\filldraw (v2n1) circle (\mandalaradius);
\draw (1.5,2.75) node { cohomology};
\draw (1.5,-1.75) node { homology};
\draw (-1.5,0.5) node[rotate=90] { left};
\draw (4.5,0.5) node[rotate=-90] { right};
\draw (v00) node[left] {$\upphi_{\mr{L}}'$};
\draw (v30) node[right] {$\upphi_{\mr{R}}'$};
\draw (v01) node[left] {$\uppsi_{\mr{L}}$};
\draw (v31) node[right] {$\uppsi_{\mr{R}}$};
\draw (v12) node[above left={-0.5ex}] {$\upphi_{\mr{L}}$};
\draw (v22) node[above right={-0.5ex}] {$\upphi_{\mr{R}}$};
\draw (v1n1) node[below left={-0.5ex}] {$\uppsi_{\mr{L}}'$};
\draw (v2n1) node[below right={-0.5ex}] {$\uppsi_{\mr{R}}'$};
\end{scope}
\end{tikzpicture}
\end{center}
\vspace{0ex}
\caption{Boardman's mandala}
\label{fig:mandala}
\end{wrapfigure}

\medskip
\noindent
The ring of stable operations $\mr{E}^*\mr{E}$ of a  cohomology theory $\mr{E}$ is not typically equipped with  an antiautomorphism \cite{B}*{p. 204}.
The natural conjugation on $\mr{E}_*\mr{E}$ is not  $\mr{E}_*$-linear and so does not pass to the dual $\mr{E}^*\mr{E}$.
Even in the presence of an antiautomorphism, the usual formula \cref{eqn:D''} does not yield an $\mr{E}^*\mr{E}$-action on duals, 
(see  \cref{rmk:AntihomUseless}).

\medskip
\noindent
This article is concerned with the case  $\mr{E}=\HM$,  $\mathbb{R}$-motivic  cohomology with coefficients in $\mathbb{F}_2$.  We rely on Boardman's mandala \cite{B} (see \cref{fig:mandala})   to demonstrate a method that computes the  action of the $\R$-motivic Steenrod algebra $\cAR$ on the Spanier-Whitehead  duals of those finite $\R$-motivic  spectra whose cohomology is free over $\bMR$. 

\medskip
\noindent
Let us pause to briefly discuss Boardman's mandala. Given a finite cell complex $\mr{X}$ there are eight ways in which its mod $2$ homology and cohomology interact with the Steenrod algebra and its dual. They represent the vertices of the mandala. Boardman identified the relationships between them, which represent the edges. Each edge of the mandala corresponds to a  formula. For example, the edge $\mr{D}''$ in \cref{fig:mandala} corresponds to  the formula   (see  \cite{B}*{p. 190})
\begin{equation} \label{eqn:D''} 
 \langle (\mr{D}''  \upphi_{\mr{L}}')(\alpha \otimes {\sf f}), {\sf x}\rangle = \langle {\sf f}, \upphi_{\mr{L}}' (\chi (\alpha) \otimes {\sf x}) \rangle 
 \end{equation} 
 that relates the left $\cA$-module structure on the cohomology $\mr{H}^*(\mr{X})$ with that of the left $\cA$-module structure on the homology of $\mr{X}$.  
 However, not all edges of the mandala exist for a general cohomology theory $\mr{E}$ (\cite{B}*{Section~6}).

 \medskip
\noindent
When $\mr{H}^{\star}(\mr{X}):= [\mr{X}, \HM]^{\star}$ is free and finitely generated over $\bMR$,  $ \mr{H}_{\star}(\mr{X})$ is the $\bMR$-linear dual of $\mr{H}^{\star}(\mr{X})$, as the relevant universal  coefficient spectral sequence collapses. 
  Consequently, the work in \cite{B} relates  the left action of $\cA^{\R}$  on $\mr{H}^{\star}(\mr{X})$ as well as the left action of $\cAR$  on $\mr{H}_\star(\mr{X})$, to the  $\cAR_{\star}$-comodule structure on $\mr{H}^{\star}(\mr{X})$ (see \cref{prop:DualizingComodule}, \cref{prop:RelatingModuleComodule} and \cref{prop:ComodFromAMod}). These relations are  the  green dashed  edges in  \cref{fig:mandala}. As a result, one deduces the left $\cAR$-module  structure on $\mr{H}_{\star}(\mr{X})$ from that of $\mr{H}^{\star}(\mr{X})$ without resorting to an antiautomorphism (unlike \eqref{eqn:D''}). 

\medskip
\noindent
Our main application is concerned with identifying the  $\R$-motivic spectra in the class $\cAR_1$ introduced in \cite{BGLtwo}. Each spectrum in $\cAR_1$ is a realization of some $\cAR$-module structure on the subalgebra $\cAR(1):= \bMR\langle \mr{Sq}^1, \mr{Sq}^2 \rangle  \subset \cAR$ (see \cref{fig:motivicA1}). In the classical case,  Davis and Mahowald \cite{DM} showed that the subalgebra $\cA(1)$ of the Steenrod algebra admits four different left $\cA$-module structures, of which two  are self-dual (see also \cite{BEM}*{Remark~1.1}). In \cite{BGLtwo}, we showed that $\cAR(1)$ admits 128 different $\cAR$-module structures. In this paper, we show:

\begin{thm} \label{thm:16}
Among the 128 different $\cAR$-module structures on $\cAR(1)$, only 16 are self-dual.
\end{thm}

\begin{rmk}
In \cite{BGLtwo} we showed that every $\cAR$-module structure on $\cAR(1)$ can be realized as a finite $\R$-motivic spectrum, but we do not know if they are unique. Hence, the spectra realizing a self-dual $\cAR$-module structure on $\cAR(1)$ may not be Spanier-Whitehead self-dual. 
\end{rmk} 

\noindent
Davis and Mahowald also showed \cite{DM} that each realization of $\cA(1)$ is the cofiber of a self-map of the spectrum $\cY:= \mathbb{S}/2 \smsh \mathbb{S}/\eta$, where $\eta$ is the first Hopf element in the stable stems.  In the $\R$-motivic stable stems, both $2$ and $\hsf$ in $\pi_{0,0}(\mathbb{S}_\R)$ are lifts of $2 \in \pi_0(\mathbb{S})$ in the classical stable stems, and  $\eta_{1,1} \in \pi_{1,1}(\mathbb{S}_\R)$ is the unique lift of $\eta$ in  bidegree $(1,1)$ (up to a unit). This results in two different $\R$-motivic lifts of $\mathcal{Y}$, namely \[ \text{$\cYR_{(2,1)} = \mathbb{S}_\R/2 \smsh \mathbb{S}_\R/\eta_{1,1}$ and $\cYR_{(\hsf,1)} = \mathbb{S}_\R/\hsf \smsh \mathbb{S}_\R/\eta_{1,1}$}. \]
We showed in \cite{BGLtwo}*{Theorem~1.8} that each $\cAR$-module structure on $\cAR(1)$ can be realized as the cofiber of a map between these $\R$-motivic lifts of $\cY$. Here we show:

\begin{thm} \label{thm:8+8}
Of the self-dual $\cAR$-module structures on $\cAR(1)$, 8 can be realized as the cofiber of a self-map on $\cYR_{(2,1)}$ and 8 as the cofiber of a self-map on $\cYR_{(\hsf,1)}$.
\end{thm}

\begin{notn} 
In all diagrams depicting modules over the Steenrod algebra, (i.e. in \cref{fig:motivicJoker}, \cref{fig:motivicA1}, and \cref{fig:HPhiA1}), a dot $\bullet$ represents a rank one free module over the coefficient ring,
black vertical lines indicate the action of $\Sq^1$, blue curved lines indicate the action of $\Sq^2$, and red  bracket-like lines represent the action of $\Sq^4$. 
A label on an edge represents that the operation hits that multiple of the generator. 
For example, in \cref{fig:motivicJoker}, $\Sq^2({\sf x}_{2,1})$ is $\mtau\cdot {\sf x}_{4,1}$ and $\Sq^4({\sf x}_{2,1})$ is $\rho^2\cdot {\sf x}_{4,1}$.
\end{notn}

\subsection*{Acknowledgements}
We thank Agnes Beaudry, Mike Hill, Clover May, Sarah Petersen, Liz Tatum, and Doug Ravenel for a stimulating conversation at the conference, Homotopy Theory in honor of Paul Goerss, held at Northwestern University in March 2023.
We also thank William Balderrama for an illuminating conversation, and we thank Dan Isaksen for pointing out a typo.

\section{A  review of the $\R$-motivic Steenrod algebra and its dual}

\noindent 
In \cite{V}, Voevodsky defined the motivic Steenrod operations 
$\Sq^n$, for $n\geq 0$, and gave a complete description 
of the $\R$-motivic Steenrod algebra $\cAR$. It is free as 
a left module
over the  $\R$-motivic homology of a point, 
\begin{equation} \label{M2}
 \bMR := \pi_{\star}^\R \HM \iso \F_2 [ \mtau,\mrho], 
 \end{equation}
  where the element $\mtau$ is in bidegree $\star = (0,-1)$, and $\mrho$ is in bidegree $\star = (-1,-1)$. 

\medskip
\noindent 
The subalgebra $\bMR \subset \cAR$ is not central, and therefore $\cAR$ has two $\bMR$-module structures, one given by left multiplication and the other by right multiplication.
The $\R$-motivic dual Steenrod algebra $\cAR_{\star}$ is defined to be the (left) $\bMR$-linear dual of $\cAR$;
it inherits an $\bMR$-module structure, which we call the left action.
The right $\bMR$-action on $\cAR$ also induces an action of $\bMR$ on $\cAR_{\star}$, which we call the right action of $\bMR$ on $\cAR_{\star}$ (see \cite{V}*{p. 48})\footnote{ Since $\bMR$ is commutative, there is no meaningful distinction between ``left'' and ``right'' actions. The adjectives are merely a bookkeeping device.}. These
correspond to the left and the right unit 
\[
\begin{tikzcd}
 \upeta_{\mr{L}} , \upeta_{\mr{R}} \colon \bMR \ar[rr] && \cAR_{\star}
 \end{tikzcd}
 \] 
 of the Hopf algebroid $(\bMR, \cAR_{\star})$.  Explicitly,
\begin{equation} \label{eqn:dualdescribe}
\cAR_\star \iso \frac{\bMR[\btau_0,\btau_1,\btau_2,\dots,\bxi_1,\bxi_2,\dots]}{\btau_n^2=\mtau \bxi_{n+1} + \mrho \btau_0 \bxi_{n+1} + \mrho \btau_{n+1}}
\end{equation}
with $\upeta_{\mr{L}}(\rho) = \upeta_{\mr{R}}(\rho) = \rho$, $\upeta_{\mr{L}}(\tau) = \tau$  and  $ \upeta_{\mr{R}}(\tau) = \tau + \rho \uptau_0$. The comultiplication 
\begin{equation} \label{eqn:DualDiag}
\begin{tikzcd}
\Delta: \cAR_{\star} \rar  &  \cAR_{\star} \underset{\bMR}{\otimes} \cAR_{\star}
\end{tikzcd}
\end{equation}
is given by
\begin{itemize}
\item $ \Delta( \upxi_n ) = \sum_{i=0}^n \bxi_{n-i}^{2^i} \tens \bxi_i $, and 
\item $ \Delta( \uptau_n ) = \btau_n \tens 1 + \sum_{i=0}^n \bxi_{n-i}^{2^i} \tens \btau_{n-i} $, 
\end{itemize}
for all $n \in \mathbb{N}$, where $\xi_0$ is the unit 1. 
The conjugation map 
$
\begin{tikzcd}
\conj\colon \cAR_\star \rar & \cAR_\star
\end{tikzcd}
 $
of the Hopf algebroid structure sends
\begin{itemize}
\item $\conj(\mrho) = \mrho$,
\item $\conj(\mtau) = \mtau + \mrho \btau_0$,
\item $\conj(\bxi_n) = \displaystyle\sum_{i=0}^{n-1} \bxi_{n-i}^{2^i} \conj(\bxi_i)$, and 
\item $\conj(\btau_n) = \btau_n + \displaystyle\sum_{i=0}^{n-1} \bxi_{n-i}^{2^i} \conj(\btau_i)$. 
\end{itemize}
\begin{rmk} The coproduct $\Delta$ in  \eqref{eqn:DualDiag} is an $\bMR$-bimodule map. 
\end{rmk}
\begin{rmk}
The conjugation is not a map of left $\bMR$-modules. In fact, it 
interchanges the left and right $\bMR$-module structures
on $\cAR_{\star}$. 
\end{rmk}
\subsection{Kronecker product} \ 

\medskip
\noindent
The $\R$-motivic Kronecker product  is a natural pairing between  $\R$-motivic homology and cohomology 
  which is constructed as follows: If $\varphi: \mr{X} \longrightarrow \Sigma^{{\sf i},{\sf j}}  \HM$ represents the class $[\varphi] \in \mr{H}^{\star}(\mr{X})$ and ${\sf x} : \Sigma^{{\sf m},{\sf n}} \mathbb{S}_\R \longrightarrow  \HM \smsh  \mr{X}$ represents $[{\sf x}] \in \mr{H}_{{\sf m},{\sf n}} (\mr{X})$, then the composition 
  \[ 
  \begin{tikzcd}
 \Sigma^{{\sf m},{\sf n}} \mathbb{S}_\R \rar["\sf x"] &  \HM \smsh \mr{X} \rar["\id \smsh \varphi"] & \Sigma^{{\sf i},{\sf j}} \HM \smsh \HM \rar &\Sigma^{{\sf i},{\sf j}} \HM 
  \end{tikzcd}
  \]
is the element $\langle {\sf x}, \varphi  \rangle \in \pi_{\star} (\HM) \cong  \bMR$. 

\medskip
\noindent
The Kronecker pairing leads to a homomorphism 
\begin{equation} \label{eqn:n}
\begin{tikzcd}
 \mathfrak{n}: \mr{H}^{\star}(\mr{X}) \rar & \mr{Hom}_{\bMR}(\mr{H}_{\star}(\mr{X}), \bMR),
 \end{tikzcd}
\end{equation}
where $\mathfrak{n}(\varphi)({\sf x}) = \langle {\sf x}, \varphi \rangle$. 

\begin{rmk} \label{rmk:niso}
When $\mr{H}_{\star} (\mr{X})$ is free and finitely generated as an  $\bMR$-module,  the map $\mathfrak{n}$ in \eqref{eqn:n} is an isomorphism. Consequently, elements in $\mr{H}^{\star}(\mr{X})$ can be identified with linear maps from $\mr{H}_{\star}(\mr{X})$, and the Kronecker product  is simply the  evaluation of functionals.   
\end{rmk}

\begin{notn} 
\label{notn:otimesLeftRight}
Since both $\cAR$ and $\cAR_{\star}$ have a left and a right action of $\bMR$,  let  $\cAR \otimes_{\bMR}^{\mathrm{left}} \cAR_\star$ (likewise $\cAR \otimes_{\bMR}^{\mathrm{right}} \cAR_\star$)   denote the tensor product of left (likewise right) $\bMR$-modules. 
\end{notn}

\begin{rmk}
When $ \mr{X}$ is $\HM$,
the Kronecker product is a map of left $\bMR$-modules  $\cAR_{\star} \otimes_{\bMR}^{\mathrm{left}} \cAR \to  \bMR$. 
\end{rmk}

\subsection{The Milnor basis} \label{subsec:Milnor}\ 

\medskip
\noindent 
The dual Steenrod algebra $\mr{H}_\star(\HM) \cong \cAR_{\star}$ is free and degree-wise finitely generated  as an $\bMR$-module. Consequently,  the natural map  of \eqref{eqn:n}  gives an isomorphism 
\begin{equation} \label{iso:dualSteenrod}
  \cAR \cong \mr{Hom}_{\bMR}(\cAR_{\star}, \bMR) 
  \end{equation}
of left $\bMR$-modules. Taking advantage of the above isomorphism, Voevodsky \cite[$\mathsection$13]{V} defines the Milnor basis of the $\R$-motivic Steenrod algebra using the monomial basis of the dual Steenrod algebra \eqref{eqn:dualdescribe}.  

\begin{table}[ht]
\begin{center}
\begin{tabular}{|L|L||L|L|L|}
\toprule
\text{degree} &  {\sf x} \in \cAR_\star & \conj ({\sf x}) & {\sf x}^{\ast} \in \cA^\R & \text{${\sf x}^{\ast}$ in terms of $\mathcal{G}$} \\
\hline \hline
(0,0) & 1 & 1 & 1 & 1 \\
\hline
(1,0) & \btau_0 & \btau_0 & \mathcal{Q}_0 & \Sq^1 \\
\hline
(2,1) & \bxi_1 & \bxi_1 & \mathcal{P}^1 & \Sq^2 \\
\hline
(3,1) & \btau_0 \bxi_1 & \btau_0 \bxi_1 & \mathcal{Q}_0 \mathcal{P}^1 & \Sq^1 \Sq^2 \\
\hline
(3,1) & \btau_1 & \btau_1 + \btau_0 \bxi_1 & \mathcal{Q}_1 & \Sq^1 \Sq^2 + \Sq^2 \Sq^1 \\
\hline
(4,1) & \btau_0 \btau_1 & \btau_0 \btau_1 + \mtau \bxi_1^2  & \mathcal{Q}_0 \mathcal{Q}_1 & \Sq^1 \Sq^2 \Sq^1 \\
 & & + \mrho \btau_0 \bxi_1^2 + \mrho \btau_1 \bxi_1 & &  \\
 \hline
(4,2) & \bxi_1^2 & \bxi_1^2 & \mathcal{P}^2 & \Sq^4 \\
\hline
(5,2) & \btau_0 \bxi_1^2 & \btau_0 \bxi_1^2 & \mathcal{Q}_0 \mathcal{P}^2 & \Sq^1 \Sq^4 \\
\hline
(5,2) & \btau_1 \bxi_1 & \btau_1 \bxi_1 + \btau_0 \bxi_1^2 & \mathcal{Q}_1 \mathcal{P}^1 & \Sq^1 \Sq^4 + \Sq^4 \Sq^1 \\
\hline
(6,2) & \btau_0 \btau_1 \bxi_1 & \btau_0 \btau_1 \bxi_1 + \mtau \bxi_1^3  & \mathcal{Q}_0 \mathcal{Q}_1 \mathcal{P}^1 & \Sq^1 \Sq^4 \Sq^1 \\
 & & + \mrho \btau_0 \bxi_1^3 + \mrho \btau_1 \bxi_1^2 & & \\
 \hline
(6,3) & \bxi_1^3 & \bxi_1^3 & \mathcal{P}^3 & \Sq^2 \Sq^4 + \mtau \Sq^1 \Sq^4 \Sq^1 \\
\hline
(6,3) & \bxi_2 & \bxi_2 + \bxi_1^3 & \mathcal{P}^{(0,1)} & \Sq^2 \Sq^4 + \Sq^4 \Sq^2 \\
\hline
(7,3) & \btau_2 & \btau_2 + \btau_1 \bxi_1^2 & \mathcal{Q}_2 & \Sq^1 \Sq^2 \Sq^4 + \Sq^1 \Sq^4 \Sq^2 \\
& & + \btau_0 \bxi_2 + \btau_0 \bxi_1^3  & & + \Sq^2 \Sq^4 \Sq^1 + \Sq^4 \Sq^2 \Sq^1 \\
\hline
(7,3) & \btau_0 \bxi_1^3 & \btau_0\bxi_1^3 & \mathcal{Q}_0 \mathcal{P}^3 & \Sq^1 \Sq^2 \Sq^4 + \mrho \Sq^1 \Sq^4 \Sq^1 \\
\hline
(7,3) & \btau_0 \bxi_2 & \btau_0 \bxi_2 + \btau_0 \bxi_1^3 & \mathcal{Q}_0 \mathcal{P}^{(0,1)} & \Sq^1 \Sq^2 \Sq^4 + \Sq^1 \Sq^4 \Sq^2 \\
\hline
(7,3) & \btau_1 \bxi_1^2 & \btau_1 \bxi_1^2 + \btau_0 \bxi_1^3 & \mathcal{Q}_1 \mathcal{P}^2 & \Sq^1 \Sq^2 \Sq^4 + \mrho \Sq^1 \Sq^4 \Sq^1 \\
 & & & & + \Sq^2 \Sq^4 \Sq^1\\
 \hline
(8,4) & \bxi_1^4 & \bxi_1^4 & \mathcal{P}^4 & \Sq^8 \\
\hline
(8,4) & \bxi_1 \bxi_2 & \bxi_1 \bxi_2 + \bxi_1^4 & \mathcal{P}^{(1,1)} & \Sq^2 \Sq^4 \Sq^2 + \mtau \Sq^1 \Sq^2 \Sq^4 \Sq^1 \\ \hline
 \end{tabular}
\captionof{table}{ The Milnor basis in low degrees}
\label{tbl:MilnorBasis}
\end{center}
\end{table}

 \medskip
\noindent
For finite sequences  $\mr{E} = ({\sf e}_0, {\sf e}_1, \dots, {\sf e}_m )$ and $\mr{R} =({\sf r}_1,  \dots, {\sf r}_n)$ of non-negative integers, let 
$\uprho(\mr{E}, \mr{R})$ denote the element in $\cAR$ dual to  the monomial 
  \[
  \uptau(\mr{E}) \upxi(\mr{R}) := \prod_{{\sf i} \geq 0} \uptau_{{\sf i}}^{{\sf e}_{\sf i}} \prod_{{\sf j} \geq 1} \upxi_{{\sf i}}^{{\sf r}_{\sf i}} 
  \] 
  in   $\cAR_{\star}$.
 It is standard practice to set $\mathcal{P}^{\mr{R}} := \uprho ( {\bf 0}, \mr{R})$ and $\mathcal{Q}^{\mr{E}} := \uprho (\mr{E}, {\bf 0})$. 
 Moreover, $\mathcal{Q}_i$ is shorthand for the dual to $\uptau_i$.

 \medskip
\noindent
 In \cref{tbl:MilnorBasis}, we record, for each monomial  $\uptau(\mr{E}) \upxi(\mr{R})\in \cAR_{\star}$ in low degree, its image under the conjugation $\conj$ and  its dual element in $\cAR$, both in terms of the Milnor basis as well as in terms of the generators $\mathcal{G}:=  \{ \Sq^{2^k}: k \geq 1 \}$.
  The latter description will be used in \cref{subsec:prelimex} and \cref{sec:A1}.   

\medskip
\noindent
A number of these descriptions in terms of $\mathcal{G}$ can be found in \cite{V}. For example, see \cite{V}*{Lemma~13.1 and Lemma~13.6}. The Adem relations (see \cite{BGLtwo}*{Appendix~A}) are another useful tool. 
For example, the Adem relation $\Sq^2 \Sq^4 = \Sq^6 + \mtau \Sq^5 \Sq^1$ leads to the description for $P^3=\Sq^6$.
The formula for $\mathcal{P}^{(0,1)}$ follows from \cite{K}*{(6)}.
Finally, the formula for $\mathcal{P}^{(1,1)}$ can be deduced from expressing $\Sq^6 \Sq^2$ in terms of the Milnor basis. This can be done by evaluating the formula \cite{V}*{(12.9)}
\[
	\langle {\sf x}, \varphi \psi \rangle = \sum \left\langle {\sf x}', \varphi  \upeta_{\mr{R}}  \big( \langle {\sf x}'', \psi \rangle \big) \right\rangle, \qquad \Delta({\sf x}) = \sum {\sf x}' \otimes {\sf x}''
\]
at $\varphi=\Sq^6$, $\psi=\Sq^2$, and $\sf x$ monomials in low degree. This shows that $\Sq^6 \Sq^2$ is the sum $\mathcal{P}^{(1,1)} + \tau \mathcal{Q}_0 \mathcal{Q}_1 \mathcal{P}^2$.

 \section{Dualizing $\cA^\R$-modules} 
 
 \noindent
 For any $\R$-motivic spectrum $\mr{X}$, its Spanier-Whitehead dual is the function spectrum $\mr{D}\mr{X}:= \mr{F} (\mr{X}, \mathbb{S}_{\R})$. The goal of this section is to identify the $\cAR$-module structure $\mr{H}^{\star}(\mr{DX})$ given the $\cAR$-module structure on $\mr{H}^{\star}(\mr{X})$ under the following assumption. 
 
  \begin{assumpt}\label{FinitenessAssumption}
Let $\mr{X}$ be a finite $\R$-motivic spectrum such that its homology $\mr{H}_{\star}(\mr{X})$ is free over $\bMR$. 
\end{assumpt}

\noindent
\begin{notn} For an $\bMR$-module $\bN$ let  
 \[ \bN^{\dual} := \mr{Hom}_{\bMR}(\bN, \bMR)  \]
 be the set of $\bMR$-linear functionals.
\end{notn}
\subsection{From $ \uppsi_{\mr{L}}$ to $\upphi_{\mr{L}}'$ } \label{sec:IdentDual} Recall that $\rH^\star(\mr{X})$ is naturally a left $\cAR$-module. 
We will also use an $\cAR_\star$-comodule structure on $\rH^{\star}(\mr{X})$
\begin{equation} \label{cohomcomodule}
\begin{tikzcd}
	\uppsi_{\mr{L}} \colon \rH^\star(\mr{X}) \ar[rr] && \cAR_\star \otimes_{\bMR} \rH^\star(\mr{X})
\end{tikzcd}
\end{equation}
which can be constructed  as follows.

\medskip
\noindent
First, note that $\cAR_\star$ is free as a right $\bMR$-module with basis $\mathcal{B}$ given by the conjugate of any left $\bMR$-module basis. Then we have a splitting
\[
	\HM \smsh \HM \simeq \bigvee_{\mathcal{B}} \HM
\]
as right $\HM$-modules. Define a map of motivic spectra $\uppsi$ as the composite
\[
\begin{tikzcd}
	\uppsi \colon \HM \iso \HM \smsh \mathbb{S}_\R \ar[rr, "\id \smsh \upiota"] &&  \HM \smsh \HM \simeq \bigvee_{\mathcal{B}} \HM,
\end{tikzcd}
\]
where $\upiota$ is the unit map of $\HM$. For any finite motivic spectrum, the map $\uppsi$ induces the map $\uppsi_{\mr{L}}$  (see \cite{B}*{Theorem~2.9(b)}) giving $\rH^\star(\mr{X})$ the structure of an $\cAR_\star$-comodule 
as explained in \cite{B}*{Section~6}.  Further,  Boardman showed that: 
\begin{prop}\cite{B}*{Lemma~3.4}\label{prop:DualizingComodule}
Let $\bN$ be a left $\cAR_\star$-comodule. Then $\bN^\dual$ inherits a left $\cAR$-module structure 
\[ 
\begin{tikzcd}
\upphi_{\mr{L}}\colon \cAR \otimes_{\bMR} \bN^\dual \rar & \bN^\dual
\end{tikzcd}
\]
 via the formula
\begin{equation}\label{eq:DualizingComodule}
	(\varphi \cdot \uplambda) (n) = (\varphi \otimes \uplambda) \uppsi_{\mr{L}}(n)
\end{equation}
for $\varphi \in \cAR$, $\uplambda \in \bN^\dual$, and $n \in \bN$.
\end{prop}

\begin{rmk}
If $\uppsi_{\mr{L}}(n) = \sum_i a_i \otimes n_i$, for $a_i \in \cAR_\star$ and $n_i \in \bN$, then \eqref{eq:DualizingComodule} can be rewritten as
\begin{equation}
\label{eq:DualizingComodulePractical}
	(\varphi \cdot \uplambda) (n) = \sum_i \varphi \Big(a_i \upeta_{\mr{R}}\big(\uplambda(n_i)\big)\Big). 
\end{equation}
\end{rmk}

\medskip
\noindent
 Combining \cref{prop:DualizingComodule} with the following result,  one can deduce the left $\cAR$-module structure on $\mr{H}^\star(\mr{DX})$ ($\upphi_{\mr{L}}'$ in \cref{fig:mandala}) from the left $\cAR_{\star}$-comodule structure on $\mr{H}^\star(\mr{X})$ ($\uppsi_{\mr{L}}$ in \cref{fig:mandala}).
 
 \begin{prop}
\label{AlgSWDuals}
Suppose  $\mr{X}$ satisfies \cref{FinitenessAssumption}. 
There are isomorphisms of left $\cAR$-modules $\rH^\star (\mr{DX}) \iso ( \rH_\star( \mr{DX }) )^\dual \iso ( \rH^\star(\mr{X}) )^\dual
$.
\end{prop}
\begin{proof}
Under \cref{FinitenessAssumption} the  map  $\mathfrak{n}:\rH^\star(\mr{DX}) \longrightarrow (\rH_\star(\mr{DX}))^\dual$ defined in  \eqref{eqn:n}, is not just an isomorphism of $\bMR$-modules (see \cref{rmk:niso}), but also an isomorphism of  left $\cAR$-modules according to  \cite{B}*{Lemma~6.2}.

\medskip
\noindent
For the second isomorphism, first note that \cref{FinitenessAssumption} implies that there exists an isomorphism  
\begin{equation} \label{eqn:iso}  \rH_\star( \mr{DX} ) \iso \rH^\star(\mr{X}) 
\end{equation}
of $\bMR$-modules. By \cref{prop:DualizingComodule}, it is enough to lift \eqref{eqn:iso}  to an isomorphism of $\cAR_\star$-comodules. To this end, we first observe that the comodule structure on $\rH^\star(\mr{X})$ is induced by the map 
\[
\begin{tikzcd}
 \mr{F}(\mr{X},\HM) \iso \mr{F}(\mr{X},\HM \smsh \mathbb{S}_\R) \ar[rrr, "{\mr{F}(\mr{X},\id \smsh \upiota)}"] &&& \mr{F}(\mr{X},\HM \smsh \HM).
\end{tikzcd}
\]
(see \eqref{cohomcomodule} or \cite{B}*{Theorem~5.4})).
The result then follows from the commutativity of the diagram
\[\begin{tikzcd}
\HM \smsh \mr{F}(\mr{X},\mathbb{S}_\R) \ar[r] \ar[equals,d] & \mr{F}(\mr{X},\HM) \ar[equals,d] \\
\HM \smsh \mathbb{S}_\R \smsh \mr{F}(\mr{X},\mathbb{S}_\R) \ar[r] \ar[d,"\id\smsh \upiota \smsh \id"] & \mr{F}(\mr{X},\HM\smsh \mathbb{S}_\R) \ar[d,"(\id\smsh \upiota)_*"] \\ 
\HM \smsh \HM \smsh \mr{F}(\mr{X},\mathbb{S}_\R) \ar[r] & \mr{F}(\mr{X},\HM \smsh \HM),
\end{tikzcd}\]
where the horizontal maps are evaluation at $\mr{X}$.
\end{proof}

\subsection{From $\upphi_{\mr{L}}$ to $\uppsi_{\mr{L}}$} For any $\varphi \in \cAR \cong \mr{Hom}_{\bMR}(\cAR_{\star}, \bMR)$, 
let $\varphi \conj$ denote the composition 
\[
\begin{tikzcd}
 \varphi \conj : \cAR_\star   \rar["\conj"] & \cAR_\star \rar["\varphi"] & \bMR, 
 \end{tikzcd}
 \]
which is a right $\bMR$-module map as the conjugation $\conj$ is an isomorphism from the right $\bMR$-module structure to the left $\bMR$-module structure of $\cAR_{\star}$. 
 \begin{prop}
\label{prop:RelatingModuleComodule}
Let $\bN$ be a left $\cAR_\star$-comodule with coproduct $\uppsi_{\mr{L}}$. Then, for $n\in \bN$ and $\varphi \in \cAR$,
the formula
\[
	\varphi \cdot n = ( \varphi \conj \otimes \id) \uppsi_{\mr{L}}(n) 
\]
defines a left $\cAR$-module structure on $\bN$.

\end{prop}
\begin{pf}
Using the coassociativity of the coaction, the statement reduces to checking that 
\begin{equation} \label{eqn:check}
	(\varphi \psi)(\conj (a)) = \sum \varphi \Big( \conj \big( \upeta_{\mr{L}} ( \psi(\conj (a_i'))) a_i'' \big) \Big),
\end{equation}
 for $\varphi, \psi \in \cAR$ and $a \in \cAR_\star$. The formula \eqref{eqn:check} follows from combining \cite{B}*{Lemma~3.3(a)} with $\conj\circ \upeta_{\mr{L}} = \upeta_{\mr{R}}$ and  
 \[ \Delta(\conj(a)) = \sum_i \conj(a_{i}'') \otimes \conj(a_{i}') \]
 whenever  $\Delta(a) = \sum_i a_i' \otimes a_i'' $.  
\end{pf}
\begin{rmk} The right $\bMR$-module structure on $\cAR_\star$ is defined \cite{V}*{Section~12} such that   \[ a \cdot \upeta_{\mr{R}}(m)(\varphi) = a(\varphi\cdot m)\]
 for $m \in \bMR$, $a\in \cAR_\star$ and $\varphi\in \cAR$. 
This shows that the evaluation pairing   defines a map  \[ 
 \begin{tikzcd}
 \cAR \otimes_{\bMR}^{\mathrm{right}} \cAR_\star \ar[rr] && \bMR
 \end{tikzcd}
 \] of $\bMR$-bimodules, where the left $\bMR$-module structure on $\cAR \otimes_{\bMR}^{\mathrm{right}} \cAR_\star$ is obtained via the left action on $\cAR$, and the right $\bMR$-module structure via the left action on $\cAR_\star$. Consequently, the left action constructed in \cref{prop:RelatingModuleComodule} can be described as the composition $\upphi_{\mr{L}}$ in the diagram
 \[ 
 \begin{tikzcd}
 \cAR\otimes_{\bMR} \bN \ar[rr, "\id \otimes \uppsi_{\mr{L}}"]  \ar[rrrdd, dashed, bend right = 20, "\upphi_{\mr{L}}"']  && \cAR\otimes_{\bMR}(\cAR_\star\otimes_{\bMR} \bN) \ar[r, cong] & (\cAR\otimes_{\bMR} \cAR_\star) \otimes_{\bMR} \bN \ar[d, "\id \otimes \conj \otimes \id"']  \\ 
  &&& (\cAR \otimes_{\bMR}^{\mathrm{right}} \cAR_\star) \otimes_{\bMR} \bN \ar[d, "\mathrm{eval}\otimes \id"'] \\
 &&& \bN \cong \bMR\otimes_{\bMR} \bN .
 \end{tikzcd}
 \]
Note that while $\conj$ is not a right $\bMR$-module map, the composition 
\[
\begin{tikzcd}
\cAR \otimes_{\bMR} \cAR_\star \ar[rr ,"\id \otimes \conj"] && \cAR \otimes_{\bMR}^{\mathrm{right}} \cAR_\star \ar[r,"\mathrm{eval}"] &  \bMR
\end{tikzcd}
\]
is a map of $\bMR$-bimodules.
\end{rmk}
\medskip
\noindent 
 If we set $\bN = \rH^\star(\mr{X})$, i.e. the cohomology of a finite spectrum $\mr{X}$ with the $\cA_{\star}$-comodule structure of \eqref{cohomcomodule},   \cref{prop:RelatingModuleComodule} recovers the usual $\cAR$-module structure on $\rH^\star(\mr{X})$ (see \cite{B}*{Lemma~6.3}). Our next result reverse-engineers \cref{prop:RelatingModuleComodule} to obtain a formula that calculates the  $\cAR_\star$-comodule on $\mr{H}^{\star}(\mr{X})$ ($\uppsi_{\mr{L}}$ in \Cref{fig:mandala}) from the $\cAR$-module on $\mr{H}^{\star}(\mr{X})$ ($\upphi_{\mr{L}}$ in \Cref{fig:mandala}). 
 
 \medskip
\noindent 
Let $\mathscr{B}$ be the monomial basis of  the left $\bMR$-module structure on  $\cAR_{\star}$ (as in  \cref{subsec:Milnor}). For simplicity, let ${\bf b}_i$ denote the elements of $\mathscr{B}$, and let ${\bf B}^i \in \cAR$ be the dual basis in the following result.

\begin{prop}
\label{prop:ComodFromAMod}
Let $\bN$ be a left $\cAR_\star$-comodule with coaction map $\uppsi_{\mr{L}}$. 
Then $\uppsi_{\mr{L}}$ is related to $\upphi_{\mr{L}}$ using the formula
\[
	\uppsi_{\mr{L}}(n) = \sum_i c ({\bf b}_i) \otimes ({\bf B}^i \cdot n),
\]
where $\cdot$ is the action of $\cAR$ on $\bN$ constructed using \cref{prop:RelatingModuleComodule}.
\end{prop}

\begin{pf}
Since $\{c({\bf b}_i)\}$ is a basis for $\cAR_\star$ as a free right $\bMR$-module, it follows that there is a unique expression $\uppsi_{\mr{L}}(n) = \sum_i c({\bf b}_i) \otimes n_i$ for appropriate elements $n_i$. On the other hand,  
\begin{eqnarray*}
	{\bf B}^k \cdot n &=& ({\bf B}^k \conj \otimes \id) \uppsi_{\mr{L}}(n) \\
	 &=& \sum_i {\bf B}^k  \conj (\conj({\bf b}_i)) \otimes n_i \\
	 &=& \sum_i {\bf B}^k  ({\bf b}_i) \otimes n_i \\
	  &=& n_k 
\end{eqnarray*}
by \cref{prop:RelatingModuleComodule}.
\end{pf}
\medskip

\subsection{Preliminary examples}  \label{subsec:prelimex} We now demonstrate the usefulness of \cref{prop:DualizingComodule}, \cref{prop:RelatingModuleComodule}, and \cref{prop:ComodFromAMod} by identifying the  $\cAR$-module structure on $\mr{H}^{\star}(\mr{DX})$, for a few well-known finite $\R$-motivic finite complexes $\mr{X}$.

\begin{notn} In the following examples, the $\R$-motivic spectrum $\mr{X}$ will satisfy \cref{FinitenessAssumption}. In particular, $\mr{H}^{\star}(\mr{X})$ will be a free $\bMR$-module. By ${\sf x}_{\sf i,j}$, we will denote an element of its $\bMR$-basis which lives in  cohomological bidegree $({\sf i}, {\sf j})$. By $\hat{\sf x}_{{\sf i}, {\sf j}}$, we will denote an element of $(\mr{H}^{\star}(\mr{X}))^{\dual}$ dual to ${\sf x}_{\sf i,j}$.  Note that the bidegree of $\hat{\sf x}_{\sf i,j}$ is $(-{\sf i}, -{\sf j})$ under the isomorphism $(\mr{H}^{\star}(\mr{X}))^{\dual} \iso \mr{H}^{\star}(\mr{DX}) $. 
\end{notn}

\begin{eg}[The $\R$-motivic mod $2$ Moore spectrum]
 As an $\bMR$-module,  $\rH^\star(\mathbb{S}_\R/2)$ has  generators  ${\sf x}_{0,0}$ and ${\sf x}_{1,0}$. The $\cAR$-module structure is then determined by the relations  \[ \Sq^1({\sf x}_{0,0}) = {\sf x}_{1,0}, \   \Sq^2({\sf x}_{0,0}) = \mrho {\sf x}_{1,0}.\]
By 
\cref{prop:ComodFromAMod}, we get 
\[
\uppsi_{\mr{L}}({\sf x}_{1,1}) = 1\otimes {\sf x}_{1,1},	\ \uppsi_{\mr{L}}({\sf x}_{0,0}) = 1\otimes {\sf x}_{0,0} + \btau_0 \otimes {\sf x}_{1,0} + \mrho \bxi_{1} \otimes {\sf x}_{1,0},  
\]
 which determines the  $\cAR_\star$-comodule structure on $\mr{H}^{\star}(\mathbb{S}_\R/2)$. 
Then we apply \cref{prop:DualizingComodule}, in particular \eqref{eq:DualizingComodulePractical}, to obtain
 \[ 
 \Sq^1(\hat{\sf x}_{1,0}) = \hat{\sf x}_{0,0}, \  \Sq^2(\hat{\sf x}_{1,0}) = \rho \hat{\sf x}_{0, 0},
 \]
 which shows  $\Big( \rH^\star(\mathbb{S}_\R/2) \Big)^\dual \iso \Sigma^{-1} \rH^\star(\mathbb{ S}_\R/2)$ as $\cAR$-modules. 
 This aligns with the fact that $D(\mathbb{S}_\R/2)$ is equivalent to $\Sigma^{-1} \mathbb{S}_\R/2$.
\end{eg}

\begin{eg}[$\R$-motivic mod ${\sf h}$ Moore spectrum] 
\label{eg:Smodh}
As a graded $\bMR$-module, $\mr{H}^{\star}(\mathbb{S}/\hsf)$ is isomorphic to $\mr{H}^{\star}(\mathbb{S}/2)$. However, they differ in their $\cAR$-module structures in that 
\[ \Sq^1({\sf x}_{0,0}) = {\sf x}_{1,0},\  \Sq^{2}({\sf x}_{0,0}) = 0 \]
 determines the $\cAR$-module structure on $\mr{H}^{\star}(\mathbb{S}/\hsf)$. By \cref{prop:ComodFromAMod} 
\[
 \uppsi_{\mr{L}}({\sf x}_{1,1}) = 1\otimes {\sf x}_{1,1}, \ \uppsi_{\mr{L}}({\sf x}_{0,0}) = 1\otimes {\sf x}_{0,0} + \btau_0 \otimes {\sf x}_{1,0},
\]
and using \eqref{eq:DualizingComodulePractical} we see that $\Big( \rH^\star(\mathbb{S}_\R/\hsf) \Big)^\dual \iso \Sigma^{-1} \rH^\star(\mathbb{ S}_\R/\hsf)$. 
This aligns with the fact that $D(\mathbb{S}_\R/\hsf)$ is equivalent to $\Sigma^{-1} \mathbb{S}_\R/\hsf$.
\end{eg}

\begin{figure}[h]
 \begin{tikzpicture}
 \begin{scope}[ thick, every node/.style={sloped,allow upside down}, scale=0.8]
\draw (0,0)  node[inner sep=0] (v0) {} -- (0,1) node[inner sep=0] (v1) {};
\draw  (-0.5,2) node[inner sep=0] (v2) {};
\draw (0,3)  node[inner sep=0] (v3) {} -- (0,4) node[inner sep=0] (v4) {};
 \draw [color=blue] (v1) to [bend right=50] (v3);
 \draw [color=blue] (v0) to [bend left=50] (v2);
 \draw [color=blue] (v2) to [bend left=50] node[pos=0.3,above={4pt},rotate=-95] {$\mtau$}  (v4);
 \draw [red] (v2) to (-1.5,2) to node[pos=0.3,above={5pt},rotate=-90] {$\mrho^2$} (-1.5,4) to (v4);
 \draw [red] (v0) to (1.15,0) to node[pos=0.6,below={5pt},rotate=-90] {$\mtau$} (1.15,4) to (v4);
\filldraw (v0) circle (2.5pt);
\filldraw (v1) circle (2.5pt);
\filldraw (v2) circle (2.5pt);
\filldraw (v3) circle (2.5pt);
\filldraw (v4) circle (2.5pt);
\draw (0,0) node[below]{$ {\sf x}_{0,0}$} (0,1) node[right]{$\ {\sf x}_{1,0} $} (-0.5,2) node[below right={-4pt},yshift={-1pt}]{$\ {\sf x}_{2,1} $};
\draw (0,4) node[above]{${\sf x}_{4,1} $} (0,3) node[right]{$\ {\sf x}_{3,1}$};
\draw (0,-1.5) node { $\mr{H}^{\star}(\mathcal{J}_{\R})$};
\end{scope}\end{tikzpicture}
\qquad
 \begin{tikzpicture}
 \begin{scope}[ thick, every node/.style={sloped,allow upside down}, scale=0.8]
\draw (0,0)  node[inner sep=0] (v0) {} -- (0,1) node[inner sep=0] (v1) {};
\draw  (-0.5,2) node[inner sep=0] (v2) {};
\draw (0,3)  node[inner sep=0] (v3) {} -- (0,4) node[inner sep=0] (v4) {};
 \draw [red] (v0) to (-1.5,0) to node[pos=0.25,above={5pt},rotate=-90] {$\mrho^2$} (-1.5,2) to (v2);
 \draw [color=blue] (v1) to [bend right=50] (v3);
 \draw [color=blue] (v0) to [bend left=50] node[pos=0.4,above={5pt},rotate=-110] {$\mtau$} (v2);
 \draw [color=blue] (v2) to [bend left=50] (v4);
\filldraw (v0) circle (2.5pt);
\filldraw (v1) circle (2.5pt);
\filldraw (v2) circle (2.5pt);
\filldraw (v3) circle (2.5pt);
\filldraw (v4) circle (2.5pt);
\draw (0,0) node[below ]{$ \hat{\sf x}_{4,1}$} (0,1) node[ right]{$\ \hat{\sf x}_{3,1} $} (-0.5,2) node[below right={-4.5pt},yshift={1pt}]{$\ \hat{\sf x}_{2,1} $};
\draw (0,4) node[above]{$\hat{\sf x}_{0,0} $} (0,3) node[right]{$\ \hat{\sf x}_{1,0}$};
\draw (0,-1.5) node {$\mr{H}^{\star}(\mr{D}\mathcal{J}_{\R})$};
\end{scope}\end{tikzpicture}
\caption{ 
The $\cAR$-module structures on the $\R$-motivic $\mathfrak{J}$oker and its dual.
}
\label{fig:motivicJoker}
\end{figure}

\begin{eg}(The $\R$-motivic $\mathfrak{J}$oker)   The  $\cAR(1)$-module of the $\R$-motivic $\mathfrak{J}$oker $\mathcal{J}_\R$ (discussed in \cite{GL}) is the quotient $\cAR(1)/\Sq^3$. 
In \cref{fig:motivicJoker}, we have displayed a particular $\cAR$-module extension of   $\cAR(1)/\Sq^3$ obtained using \cref{list}. 
Using \cref{prop:ComodFromAMod}, in conjunction with \cref{tbl:MilnorBasis}, we notice that 
\begin{eqnarray*}
\uppsi_{\mr{L}}({\sf x}_{4, 2}) &=&  1 \otimes {\sf x}_{4,2} \\
\uppsi_{\mr{L}}({\sf x}_{3, 1}) &=&  1\otimes {\sf x}_{3,1} + \btau_0 \otimes {\sf x}_{4,2} \\
\uppsi_{\mr{L}}({\sf x}_{2, 1}) &=&  1 \otimes {\sf x}_{2,1} + (\mtau \bxi_1 + \mrho \btau_0 \bxi_1 + \mrho \btau_1 + \mrho^2 \bxi_1^2) \otimes {\sf x}_{4,2} \\
\uppsi_{\mr{L}}({\sf x}_{1, 0}) &=&  1 \otimes {\sf x}_{1,0} + \upxi_1 \otimes {\sf x}_{3,1} +  \uptau_1 \otimes {\sf x}_{4,2}  \\
\uppsi_{\mr{L}}({\sf x}_{0, 0}) &=&  1 \otimes {\sf x}_{0,0} + \btau_0 \otimes {\sf x}_{1,0} + \bxi_1 \otimes {\sf x}_{2,1} + (\btau_0 \bxi_1 + \btau_1) \otimes {\sf x}_{3,1} \\
&&  +(\btau_0 \btau_1  + \mrho^2 \bxi_2 + \mrho^2 \bxi_1^3) \otimes {\sf x}_{4,2}
\end{eqnarray*}
determines the $\cAR_{\star}$-comodule structure of $\mr{H}^{\star}(\mathcal{J}_{\R})$.
Then \eqref{eq:DualizingComodulePractical} produces the $\cAR$-module structure on the dual displayed in \cref{fig:motivicJoker}.
\end{eg}

\section{Self-dual $\cAR$-module structures on $\cAR(1)$}
\label{sec:A1}
\noindent
Let  ${\sf x}_{\sf i,j}$ and ${\sf y}_{\sf i,j}$ denote the elements of the $\bMR$-basis of $\cAR(1)$ introduced in \cite[Notation 1.5]{BGLtwo} in bidegree $(\sf i,j)$. 

\begin{thm}\cite{BGLtwo}*{Theorem~1.6}  \label{list}
For every vector  
\[
	 \overline{\mr{v}} = ( \alpha_{03}, \beta_{03},\beta_{14}, \beta_{06}, \beta_{25}, \beta_{26}, \gamma_{36}) \in \F_2^7,
\]
there exists a unique isomorphism class of $\cAR$-module structures on $\cAR(1)$, which we denote by $\cAR_{ \overline{\mr{v}}}(1)$, determined by the formulas
\begin{eqnarray*}
 \Sq^4({\sf x}_{0,0}) &=& \beta_{03} ( \mrho \cdot   {\sf y}_{3,1}) +  (1+ \beta_{03} + \beta_{14}) ( \mtau \cdot   {\sf y}_{4,1}) + \alpha_{03}( \mrho \cdot  {\sf x}_{3,1}) \\
 \Sq^{4}({\sf x}_{1,0}) &=& {\sf y}_{5,2} + \beta_{14} ( \mrho \cdot   {\sf y}_{4,1}) \\
 \Sq^4({\sf x}_{2,1}) &=& \beta_{26}(\mtau \cdot {\sf y}_{6,2}) + \beta_{25} (\mrho \cdot {\sf y}_{5,2}) + \jay_{24} (\mrho^2 \cdot {\sf y}_{4,1}) \\
 \Sq^4({\sf x}_{3,1}) &=& (\beta_{25} + \beta_{26})(\mrho \cdot {\sf y}_{6,2}) \\
 \Sq^4({\sf y}_{3,1}) &=& \gamma_{36}(\mrho \cdot {\sf y}_{6,2}) \\
 \Sq^8({\sf x}_{0,0}) &=& \beta_{06} (\mrho^2 \cdot {\sf y}_{6,2}),
\end{eqnarray*}
where $\jay_{24}= \beta_{03}  \gamma_{36} + \alpha_{03}(\beta_{25} + \beta_{26}).$
Further, any $\cAR$-module whose underlying $\cAR(1)$-module is free on one generator is isomorphic to one listed above. 
\end{thm}

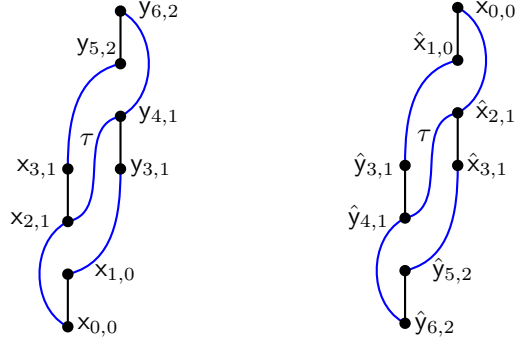
\begin{figure}[h]
 \[
\begin{tikzpicture}\begin{scope}[thick, every node/.style={sloped,allow upside down}, scale=0.7]
\draw (0,0)  node[inner sep=0] (v00) {} -- (0,1) node[inner sep=0] (v01) {};
\draw (0,2)  node[inner sep=0] (v11) {} -- (0,3) node[inner sep=0] (v12) {};
\draw (1,3)  node[inner sep=0] (v22) {} -- (1,4) node[inner sep=0] (v23) {};
\draw (1,5)  node[inner sep=0] (v33) {} -- (1,6) node[inner sep=0] (v34) {};
 \draw [color=blue] (v00) to [out=150,in=-150] (v11);
 \draw [color=blue] (v01) to [out=15,in=-90] (v22);
 \draw [color=blue] (v12) to [out=90,in=-165] (v33);
 \draw [color=blue] (v23) to [out=30,in=-30] (v34);
 \draw [color=blue] (v11) to [out=15,in=-165] node[pos=0.6,above={3pt},rotate=-85] {\textcolor{black}{$\mtau$}} (v23);
\filldraw (v00) circle (2.5pt);
\filldraw (v01) circle (2.5pt);
\filldraw (v11) circle (2.5pt);
\filldraw (v12) circle (2.5pt);
\filldraw (v22) circle (2.5pt);
\filldraw (v23) circle (2.5pt);
\filldraw (v33) circle (2.5pt);
\filldraw (v34) circle (2.5pt);
\draw (0,0) node[right]{$ {\sf x}_{0,0}$} (0,1) node[right]{ \ \ ${\sf x}_{1,0}$} (0,2) node[left,xshift={1pt}]{${\sf x}_{2,1}  \ $ } (0,3) node[left]{${\sf x}_{3,1} $  };
\draw (1,6) node[right]{$\ {\sf y}_{6,2}$} (1.3,5.3) node[left]{${\sf y}_{5,2}  \ $} (1,4) node[right]{$ \ {\sf y}_{4,1} $} (1,3) node[right]{$ {\sf y}_{3,1}$};
\end{scope}\end{tikzpicture} 
\hspace{2cm}
\begin{tikzpicture}\begin{scope}[thick, every node/.style={sloped,allow upside down}, scale=0.7]
\draw (0,0)  node[inner sep=0] (v00) {} -- (0,1) node[inner sep=0] (v01) {};
\draw (0,2)  node[inner sep=0] (v11) {} -- (0,3) node[inner sep=0] (v12) {};
\draw (1,3)  node[inner sep=0] (v22) {} -- (1,4) node[inner sep=0] (v23) {};
\draw (1,5)  node[inner sep=0] (v33) {} -- (1,6) node[inner sep=0] (v34) {};
 \draw [color=blue] (v00) to [out=150,in=-150] (v11);
 \draw [color=blue] (v01) to [out=15,in=-90] (v22);
 \draw [color=blue] (v12) to [out=90,in=-165] (v33);
 \draw [color=blue] (v23) to [out=30,in=-30] (v34);
 \draw [color=blue] (v11) to [out=15,in=-165] node[pos=0.6,above={3pt},rotate=-85] {\textcolor{black}{$\mtau$}} (v23);
\filldraw (v00) circle (2.5pt);
\filldraw (v01) circle (2.5pt);
\filldraw (v11) circle (2.5pt);
\filldraw (v12) circle (2.5pt);
\filldraw (v22) circle (2.5pt);
\filldraw (v23) circle (2.5pt);
\filldraw (v33) circle (2.5pt);
\filldraw (v34) circle (2.5pt);
\draw (0,0) node[right]{$ \hat{\sf y}_{6,2}$} (0,1) node[right]{ \ \ $\hat{\sf y}_{5,2}$} (0,2) node[left,xshift={1pt}]{$\hat{\sf y}_{4,1}  \ $ } (0,3) node[left]{$\hat{\sf y}_{3,1} $  };
\draw (1,6) node[right]{$\ \hat{\sf x}_{0,0}$} (1.3,5.3) node[left]{$\hat{\sf x}_{1,0}  \ $} (1,4) node[right]{$ \ \hat{\sf x}_{2,1} $} (1,3) node[right]{$ \hat{\sf x}_{3,1}$};
\end{scope}\end{tikzpicture}
\]
\caption{A singly-generated free $\cAR(1)$-module (on the left), and its dual (on the right). }
\label{fig:motivicA1}
 \end{figure}
 
\noindent
Using \cref{prop:ComodFromAMod}, we  calculate the $\cAR_\star$-comodule structure $\uppsi_{\mr{L}}$  on $\cAR_{\overline{\mr{v}}}(1)$: 
\begin{eqnarray*}
\uppsi_{\mr{L}}({\sf y}_{6,2}) & =&  1 \otimes {\sf y}_{6,2}\\
\uppsi_{\mr{L}}({\sf y}_{5,2}) & =&  1\otimes {\sf y}_{5,2} + \btau_0 \otimes {\sf y}_{6,2} \\
\uppsi_{\mr{L}}({\sf y}_{4,1}) & =& 1\otimes {\sf y}_{4,1} + \bxi_1 \otimes {\sf y}_{6,2} \\
\uppsi_{\mr{L}}({\sf y}_{3,1}) & =& 1\otimes {\sf y}_{3,1} + \btau_0 \otimes {\sf y}_{4,1} + (\btau_1 + \btau_0 \bxi_1 + \gamma_{36} \mrho \bxi_1^2) \otimes {\sf y}_{6,2} \\
\uppsi_{\mr{L}}({\sf x}_{3,1}) & =&  1 \otimes {\sf x}_{3,1} + \bxi_1 \otimes {\sf y}_{5,2} + (\btau_1 + (\beta_{25} + \beta_{26})  \mrho \bxi_1^2 ) \otimes {\sf y}_{6,2} \\
\uppsi_{\mr{L}}({\sf x}_{2,1}) & =& 1 \otimes {\sf x}_{2,1} + \btau_0 \otimes {\sf x}_{3,1} + (\mtau \bxi_1 + \mrho \btau_1 + \mrho \btau_0 \bxi_1 + \jay_{24} \mrho^2 \bxi_1^2) \otimes {\sf y}_{4,1} \\
 &&  + (\btau_1 + \btau_0 \bxi_1 + \beta_{25} \mrho \bxi_1^2) \otimes {\sf y}_{5,2} + (\btau_0 \btau_1 + (1 + \beta_{26}) \mtau \bxi_1^2) \otimes {\sf y}_{6,2}\\
  && + ( (1 + \beta_{25}) \mrho \btau_0 \bxi_1^2 + \mrho \btau_1 \bxi_1 + \jay_{24} \mrho^2 \bxi_2) \otimes {\sf y}_{6,2} \\
  \uppsi_{\mr{L}}({\sf x}_{1,0}) & =&1 \otimes {\sf x}_{1,0} + \bxi_1 \otimes {\sf y}_{3,1} + (\btau_1 + \beta_{14} \mrho \bxi_1^2)  \otimes {\sf y}_{4,1} + \bxi_1^2 \otimes {\sf y}_{5,2} \\
 && + (\btau_1 \bxi_1 + \gamma_{36} \mrho \bxi_1^3 + (\beta_{14} + \gamma_{36}) \mrho \bxi_2) \otimes {\sf y}_{6,2} \\
 \uppsi_{\mr{L}}({\sf x}_{0,0}) & =& 1\otimes \sf x_{0,0} + \btau_0 \otimes {\sf x}_{1,0} + \bxi_1 \otimes {\sf x}_{2,1} + (\btau_1 + \alpha_{03} \mrho \bxi_1^2 ) \otimes {\sf x}_{3,1} \\
 && + (\btau_1 + \btau_0 \bxi_1 + \beta_{03} \mrho \bxi_1^2) \otimes {\sf y}_{3,1} \\
 && + (\btau_0\btau_1 + ( \beta_{03} + \beta_{14}) \mtau \bxi_1^2 + ( \beta_{03}) \mrho \btau_0 \bxi_1^2 
 + \jay_{24} \mrho^2 \bxi_2 + \jay_{24} \mrho^2 \bxi_1^3) \otimes {\sf y}_{4,1}  \\
 && + (\btau_1 \bxi_1 + \btau_0 \bxi_1^2 + \beta_{25} \mrho \bxi_1^3 + (\alpha_{03} + \beta_{25}) \mrho \bxi_2) \otimes {\sf y}_{5,2} \\
 &&  + (\beta_{26} \mtau \bxi_1^3 + (\beta_{26} + \gamma_{36}) \mrho \btau_0 \bxi_1^3 + (\beta_{25} + \beta_{26} + \gamma_{36}) \mrho \btau_1 \bxi_1^2) \otimes  {\sf y}_{6,2} \\
 &&  + ((1 + \beta_{03} + \beta_{14} + \beta_{26}) \mtau \bxi_2 + (1 + \beta_{03} + \beta_{26} + \gamma_{36}) \mrho \btau_0 \bxi_2)  \otimes  {\sf y}_{6,2}\\
 &&  + ((1 + \alpha_{03} + \beta_{03} + \beta_{25} + \beta_{26} + \gamma_{36} ) \mrho \btau_2 + \jay_{24} \mrho^2 \bxi_1 \bxi_2) \otimes  {\sf y}_{6,2} \\
 &&    + (\btau_0 \btau_1 \bxi_1 +(\jay_{24} + \beta_{06}) \mrho^2 \bxi_1^4 ) \otimes {\sf y}_{6,2}.
\end{eqnarray*}

\medskip
\noindent
Using \eqref{eq:DualizingComodulePractical}, we get the following result, where $\hat{\sf x}_{\sf i,j}$ and  $\hat{\sf y}_{\sf i,j}$ are the elements in  $(\cAR_{ \overline{\mr{v}}}(1))^\dual$ dual to ${\sf x}_{\sf i, j}$ and ${\sf y}_{\sf i,j}$, respectively. 

\begin{thm}
\label{AmoduleDualAone}
The $\cAR(1)$-module structure on the dual $(\cAR_{ \overline{\mr{v}}}(1))^\dual$ is as displayed in the right of \cref{fig:motivicA1}.  Moreover,  its $\cAR$-module structure is determined  by
\begin{eqnarray*}
\Sq^4 (\hat{\sf y}_{6,2}) &=& (\beta_{25} + \beta_{26}) (\mrho \cdot  \hat{\sf x}_{3,1}) + (1 + \beta_{26} ) (\mtau  \cdot \hat{\sf x}_{2,1}) + \gamma_{36} (\mrho \cdot \hat{\sf y}_{3,1})  \\
 \Sq^4( \hat{\sf y}_{5,2} ) &=& \hat{\sf x}_{1,0}  + \beta_{25} (\mrho \cdot \hat{\sf x}_{2,1}) \\
 \Sq^4 ( \hat{\sf y}_{4,1} ) &=& (\beta_{03} + \beta_{14}) (\mtau \cdot \hat{\sf x}_{0,0}) +  \beta_{14} (\mrho \cdot \hat{\sf x}_{1,0})  + \jay_{24} (\mrho^2  \cdot \hat{\sf x}_{2,1}) \\
 \Sq^4 ( \hat{\sf y}_{3,1} ) &=& \beta_{03} (\mrho \cdot \hat{\sf x}_{0,0})  \\
 \Sq^4 ( \hat{\sf x}_{3,1} ) &=& \alpha_{03} (\mrho \cdot  \hat{\sf x}_{0,0}) \\
 \Sq^8 ( \hat{\sf y}_{6,2} ) &=& (\jay_{24} + \beta_{06}) (\mrho^2  \cdot \hat{\sf x}_{0,0}).
\end{eqnarray*}
\end{thm}
\begin{cor} \label{Rmotivicselfdual}
For the $\cAR$-module  $\cAR_{ \overline{\mr{v}}}(1)$,  its (regraded) dual is isomorphic to 
\[
\Sigma^{6,2}(\cAR_{ \overline{\mr{v}}}(1))^\dual \iso  
\cAR_{\updelta(\overline{\mr{v}})}(1),
\]
where $\updelta(\overline{\mr{v}}) =  (\gamma_{36},\beta_{25}+\beta_{26},\beta_{25},\jay_{24}+\beta_{06},\beta_{14},\beta_{03} + \beta_{14},\alpha_{03}).$ Thus,  $\cAR_{ \overline{\mr{v}}}(1)$ is self dual if and only if
 \begin{enumerate}
 \item \label{cond1}  $\alpha_{03} = \gamma_{36}$,
 \item \label{cond2} $\beta_{03} = \beta_{25} + \beta_{26}$, and 
 \item \label{cond3} $\beta_{14} = \beta_{25}$.
 \end{enumerate}
 \end{cor}
 \begin{rmk}
 The constant $\jay_{24}$ has a geometric significance noted in \cite[Remark 1.21]{BGLtwo}. It follows from \cref{Rmotivicselfdual} that  $\jay_{24}=0$ whenever $\cAR_{ \overline{\mr{v}}}(1)$ is self-dual. 
\end{rmk}
\begin{rmk} The underlying classical $\cA$-module structure on $\cA(1)$ is self-dual if and only if $\beta_{26} = \beta_{03} + \beta_{14}$. In the presence of \eqref{cond3}, this is equivalent to \eqref{cond2}. Thus the conditions of \cref{Rmotivicselfdual} can be thought of as the classical condition, plus conditions \eqref{cond1} and \eqref{cond3}.
\end{rmk}
\noindent
In \cite{BGLtwo}, we showed that the $\cAR$-modules $\cAR_{ \overline{\mr{v}}}(1)$ can be realized as the cohomology of an $\R$-motivic spectrum for all values of $\overline{\mr{v}}$.

\begin{cor}  \label{cor:self-map}
Suppose $\cAR_1[\overline{\mr{v}}]$ is an $\R$-motivic spectrum realizing  $\cAR_{ \overline{\mr{v}}}(1)$, and 
suppose that $\cAR_{ \overline{\mr{v}}}(1)$  is a self-dual $\cAR$-module. Then  $\cAR_1[\overline{\mr{v}}]$ is the cofiber of a $v_1$-self-map on either $\cYR_{2,1}$ or $\cYR_{\hsf,1}$.
\end{cor}

\begin{pf}
By \cite{BGLtwo}*{Theorem~1.8}, the $\R$-motivic spectrum $\cAR_1[\overline{\mr{v}}]$ is the cofiber of a $v_1$-self map on $\cYR_{2,1}$ if $\beta_{25}+\beta_{26}+\gamma_{36}=1$ and $\alpha_{03}+\beta_{03}=1$, whereas it is the cofiber of a $v_1$-self-map on $\cYR_{\hsf,1}$ if  $\beta_{25}+\beta_{26}+\gamma_{36}=0$ and $\alpha_{03}+\beta_{03}=0$.
But conditions (1) and (2) of \cref{Rmotivicselfdual} imply that $\beta_{25}+\beta_{26}+\gamma_{36}$ is equal to $\alpha_{03}+\beta_{03}$.
\end{pf}

\noindent
Our main results \cref{thm:16} and \cref{thm:8+8} follows from  \cref{Rmotivicselfdual}  and \cref{cor:self-map} respectively. 

\begin{figure}[h]
\[
\begin{tikzpicture}\begin{scope}[ thick, every node/.style={sloped,allow upside down}, scale=0.8]
\draw (0,0)  node[inner sep=0] (s0) {} -- (0,1) node[inner sep=0] (s1a) {};
\draw (1.5,2)  node[inner sep=0] (t2) {} -- (1.5,3) node[inner sep=0] (t3a) {};
\draw (3,1)  node[inner sep=0] (s1b) {} -- (3,2) node[inner sep=0] (s2) {};
\draw (4.5,3)  node[inner sep=0] (t3b) {} -- (4.5,4) node[inner sep=0] (t4) {};
 \draw [color=blue,bend right] (s0) to node[pos=0.5,above] {$\alpha_{03}$} (t2);
  \draw [color=blue,out=20,in=-160] (s0) to node[pos=0.5,below] {$\beta_{03}$} (s2);  
  \draw [color=blue,bend left=20] (s1a) to node[pos=0.4,above={-1pt}] {$\beta_{14}$}  (t3a);
   \draw [color=blue,out=150,in=-30] (s1b) to (t3a);
  \draw [color=blue,bend right] (s1b) to  node[pos=0.5,below] {$\beta_{14}$}(t3b);
   \draw [color=blue,out=30,in=210] (s2) to node[pos=0.5,below] {$\alpha_{03}$} (t4);
   \draw [color=blue,out=15,in=195] (t2) to node[pos=0.7,above] {$\beta_{03}$} (t4);
\draw [color = red] (s0) to (5.5,0) to (5.5,4) to (t4);
\filldraw (s0) circle (2.5pt);
\filldraw (s1a) circle (2.5pt);
\filldraw (t2) circle (2.5pt);
\filldraw (t3a) circle (2.5pt);
\filldraw (s1b) circle (2.5pt);
\filldraw (s2) circle (2.5pt);
\filldraw (t3b) circle (2.5pt);
\filldraw (t4) circle (2.5pt);
\draw (s0) node[left]{$ {\sf s}_0$} (0,1) node[left]{$ {\sf s}_{1a}$} (s1b) node[below]{${\sf s}_{1b} $} (s2) node[right]{$\ {\sf s}_2 $};
\draw (t2) node[left]{$ {\sf t}_2$} (t3a) node[above]{$ {\sf t}_{3a}$} (t3b) node[below right={-2pt}]{${\sf t}_{3b} $} (t4) node[above]{${\sf t}_4 $};
\draw (5.5,2) node[right]{$\beta_{06}$};
\end{scope}\end{tikzpicture}
\]
\caption{The $\cA$-module structure of  a self-dual $\rH^*(\GFix{\cA_{1}^\Ctwo [\overline{\mr{v}}]})$. }
\label{fig:HPhiA1}
\end{figure}
\noindent
\begin{rmk}
Using the Betti realization functor, \cite{BGLtwo} produced $\mr{C}_2$-equivariant realizations of analogous $\cA^{\mr{C}_2}$-modules $\cA_{\overline{\mr{v}}}^{\mr{C}_2} (1)$. Using the comparison result \cite[Theorem 1.19]{BGLtwo}, the $\cA$-module structures on  $\Phi(\cA_{1}^{\mr{C}_2} [\overline{\mr{v}}])$, the geometric fixed points of $\cA_{1}^{\mr{C}_2} [\overline{\mr{v}}]$, was identified in \cite[Figure 4.12]{BGLtwo}. In \cref{fig:HPhiA1}, we record the $\cA$-module structure on the geometric fixed points of a self-dual $\cA_{1}^{\mr{C}_2} [\overline{\mr{v}}]$. 
\end{rmk}

\appendix
\section{ On the antiautomorphism of $\cAR$}
\label{sec:Lift}

\noindent
Although Boardman \cite[$\mathsection$6]{B} pointed out that the set of $\mr{E}$-cohomology operations $[\mr{E}, \mr{E}]^{\ast}$ may not necessarily have an antiautomorphism  for a cohomology theory $\mr{E}$, we find the case of $\mr{E} = \HM$ a rather curious one. 

\medskip
\noindent
 The case of $\mr{E} = \HF$ is exceptional;  the Steenrod algebra $\cA := [\HF, \HF]_*$ is well-known to be a Hopf algebra and, therefore, equipped with an antiautomorphism 
$
 \begin{tikzcd}
 \chi:\cA \ar[r] & \cA.
 \end{tikzcd}
$
The composition of extension of scalars and Betti realization,
\[
\begin{tikzcd}
\Sp^\R \ar[rr, "\C\otimes_\R -"] && \Sp^\C \ar[rr,"\Betti"] && \Sp, 
\end{tikzcd}
\]
induces maps of Steenrod algebras
\[ 
\begin{tikzcd}
\cAR \rar[two heads, "\uppi_1"] & \cA^\C \cong \cAR/(\rho) \rar[two heads, "\uppi_2"] & \cA \cong \cAR/(\tau-1, \rho),
\end{tikzcd}
\]
where $\uppi_1$ sends $\rho$ to 0 and $\uppi_2$ sends $\tau$ to 1.

\medskip
\noindent
The antiautomorphism $\chi$ of the classical Steenrod algebra is known to lift along $\uppi_{2}$,
\[ 
\begin{tikzcd}
\cA^\C \rar["\chi^{\C}"] \dar["\uppi_2"'] & \cA^\C \dar["\uppi_2"] \\
\cA \rar["\chi"] & \cA,
\end{tikzcd}
\]
 as the $\C$-motivic Steenrod algebra
is a connected bialgebra.
However, lifting  $\chi^{\C}$ along $\uppi_1$  is less straightforward.
The dual $\R$-motivic Steenrod algebra $\cAR_\star$ is a Hopf {\it algebroid}, rather than a Hopf algebra, so that its dual is not a Hopf algebra. 

\medskip
\noindent
One feature that distinguishes $\cAR$ from $\cAC$ is the fact that $\mtau$ is not central in $\cAR$. In the following result, we use the commutators $[\mtau, \Sq^{2^n}]$ in $\cAR$ (computed using the Cartan formula \cite{V}*{Proposition~9.7}) to compute the values of a hypothetical antiautomorphism in low degrees.

\begin{prop} \label{prop:antipodeR}
Suppose that $\chi^\R \colon \cAR \rtarr \cAR$ is a ring antihomomorphism and an involution.
Then 
\begin{eqnarray*}
	\chi^\R(\mtau) &=& \mtau \\ 
	\chi^\R(\mrho) &=& \mrho \\
	 \chi^\R(\Sq^1) &=& \Sq^1 \\
	\chi^\R(\Sq^2) &=& \Sq^2 + \mrho \Sq^1 \\
	 \chi^\R(\Sq^4) &=&  \Sq^4 + \mrho \Sq^2 \Sq^1 + \mtau \Sq^1 \Sq^2 \Sq^1.
\end{eqnarray*}
\end{prop}

\begin{pf} If $\chi^\R$ is a ring antihomomorphism then 
\begin{equation} \label{chicommute}
 \chi^\R[r,s] =  [\chi^\R r,\chi^\R s] 
 \end{equation}
 in characteristic $2$.  Since $\tau$ and $\Sq^1$ are unique $\F_2$-generators in their bidegree and $\chi^\R$ is an automorphism, it follows that \[ \chi^\R (\mtau) = \mtau \qquad \text{and} \qquad \ \chi^\R(\Sq^1) = \Sq^1.\]
 For degree reasons, $\chi^\R (\Sq^2)$ must be $\Sq^2 + \varepsilon \mrho \Sq^1$, where $\varepsilon$ is either 0 or 1. 
 But the commutator $[\mtau,\Sq^2]$ is equal to $\mrho \mtau \Sq^1$. Applying \cref{chicommute}, we see that  
 \begin{eqnarray*}
 \chi^\R( \mrho \mtau \Sq^1 ) &=&  [ \chi^\R( \mtau), \chi^\R (\Sq^2) ] \\
 \Rightarrow \hspace{1.65cm} \Sq^1 \tau \rho &=& [ \mtau, \Sq^2 + \varepsilon \mrho \Sq^1 ] \\
  \Rightarrow \hspace{1cm}\mrho \mtau \Sq^1 + \mrho^2 &=& \mrho \mtau \Sq^1 + \varepsilon \mrho^2, 
 \end{eqnarray*}
and therefore,  $\varepsilon$ must be 1.

\medskip 
\noindent
Similarly, degree considerations imply that $\chi^\R (\Sq^4)$ must be of the form $\Sq^4 + \delta \mrho \Sq^1 \Sq^2 + \varepsilon \mrho \Sq^2 \Sq^1 + \lambda \mtau \Sq^1 \Sq^2 \Sq^1$. 
The commutator $[ \mtau, \Sq^4 ]$ is $\mrho \mtau \Sq^1 \Sq^2$, so we conclude that
\begin{eqnarray*}
[ \chi^\R \mtau, \chi^\R \Sq^4 ] &=& [ \mtau,  \Sq^4 + \delta \mrho \Sq^1 \Sq^2 + \varepsilon \mrho \Sq^2 \Sq^1 + \lambda \mtau \Sq^1 \Sq^2 \Sq^1] \\
	&=& (1+\lambda) \mrho \mtau \Sq^1 \Sq^2 + \lambda \mrho \mtau \Sq^2 \Sq^1 + (\delta+\varepsilon) \mrho^2 \Sq^2 + \delta \mrho^3 \Sq^1
\end{eqnarray*}
must agree with

\begin{eqnarray*}
	\chi^\R ( \mrho \mtau \Sq^1 \Sq^2 ) &=& (\Sq^2 + \mrho \Sq^1) \Sq^1 \mtau \mrho \\
	&=& \mrho \mtau \Sq^2 \Sq^1 + \mrho^2 \Sq^2, 
\end{eqnarray*}
and therefore, $\delta=0$, $\varepsilon=1$, and $\lambda=1$ as desired.
\end{pf}
\noindent 

\noindent
\cref{prop:antipodeR} suggests there might be an $\R$-motivic antiautomorphism on the subalgebra $\cA^\R(2): = \bMR\langle \Sq^1, \Sq^2, \Sq^4 \rangle \subset \cA^\R$. It seems likely that the method above can be extended to produce an antiautomorphism on all of $\cAR$. However, we leave open the question of whether or not this is possible. 
 
 \medskip
 \noindent
 On the other hand, the following remark shows that an antihomomorphism on $\cAR$ may not be directly of use in dualizing $\cAR$-modules.
 
\begin{rmk}
\label{rmk:AntihomUseless}
Note that if $\bN$ is an $\cAR$-module, then the action of $\cAR$ on $\bN$ is not $\bMR$-linear, so that, in contrast to the classical case, it does not induce a right $\cAR$-action on the dual $\bN^{\dual}$.
Even if $\cAR$ were to be hypothetically equipped with an antiautomorphism $\chi^\R$, this may not be so useful for the purpose of dualization. The reason is that the classical formula \cref{eqn:D''} does not work in this setting. More precisely, let $\bN$ be an $\cAR$-module, let $\lambda \in \bN^{\dual}$, $\varphi \in \cAR$, and $n \in {\bf N}$. Then defining a new action  $\varphi\odot \lambda$ by
\[ 
	(\varphi \odot \lambda)(n) = \lambda( \chi^\R \varphi \cdot n)
\]
does not produce an $\bMR$-linear function. For instance, consider the case $\bN=\mr{H}^{\star}(\mathbb{S}_\R/\hsf)$ from \cref{eg:Smodh}. Then $(\Sq^2\odot \hat{\sf x}_{1,0})( \mtau {\sf x}_{0,0})$ vanishes, whereas $(\Sq^2\odot \hat{\sf x}_{1,0})(  {\sf x}_{0,0})$ is equal to $\mrho$. It follows that the formula for $\Sq^2 \odot \hat{\sf x}_{1,0}$ is not $\bMR$-linear and is therefore not a valid element of $\bN^\dual$.
\end{rmk}

\bibliographystyle{amsalpha}

\end{document}